\documentclass[a4paper]{amsart}

\usepackage[utf8]{inputenc}
\usepackage[english]{babel}

\usepackage{amsmath,amssymb,amsthm}
\usepackage[all]{xy}

\newcommand{\Q}{\mathbb{Q}}
\newcommand{\Z}{\mathbb{Z}}
\newcommand{\F}{\mathbb{F}}
\newcommand{\R}{\mathbb{R}}

\DeclareMathOperator{\kar}{char}

\DeclareMathOperator{\cd}{cd}

\newcommand{\llangle}{\mathopen{\langle\langle}}
\newcommand{\rrsquare}{\mathclose{]]}}

\newtheorem{theorem}{Theorem}[section]
\newtheorem{proposition}[theorem]{Proposition}
\newtheorem{lemma}[theorem]{Lemma}
\newtheorem{corollary}[theorem]{Corollary}
\newtheorem{fact}[theorem]{Fact}
\newtheorem{definition}[theorem]{Definition}
\newtheorem{hypothesis}[theorem]{Hypothesis}

\theoremstyle{remark}
\newtheorem{remark}[theorem]{Remark}

\newcommand{\noopsort}[1]{}

\title[Defining Subrings in Finitely Generated Fields]{Defining Subrings in Finitely Generated Fields of All Characteristics}
\thanks{This is an updated version of the manuscript \cite{definingSubringsOfFinGenFieldsOfCharNotTwo}. A proof of the main result in characteristic away from two has also been announced by Pop in \cite{PopDistinguishingEveryFinitelyGeneratedFieldOfCharNotTwo}.}
\author{Philip Dittmann}
\address{KU Leuven, Afdeling Algebra, Celestijnenlaan 200b, 3001 Leuven, Belgium}
\email{philip.dittmann@kuleuven.be}

\date{\today}

\begin{document}

\begin{abstract}
  We give a construction of a large first-order definable family of subrings of finitely generated fields $K$ of any characteristic. We deduce that for any such $K$ there exists a first-order sentence $\varphi_K$ characterising $K$ in the class of finitely generated fields, i.e.~such that for any finitely generated field
$L$ we have $L \models \varphi_K$ if and only if $L \cong K$. This answers a question considered by Pop and others.
In characteristic two, our results depend on resolution of singularities, whereas they are unconditional in all other characteristics.
\end{abstract}

\maketitle

\section{Introduction}

First-order logic naturally applies to the study of fields. Consequently, it is of interest to investigate the expressive power of first-order logic in certain classes of fields.
This is well-understood in the cases of algebraically closed fields, real-closed fields and $p$-adic fields.
On the other hand, it is known by \cite{RumelyUndecidabilityGlobalFields} that in global fields, essentially due to Gödelian phenomena, first-order logic is very expressive, and the class of definable set is very complicated.
For infinite finitely generated fields, many questions about the expressive power are open; see \cite{poonenUniformDefnsInFinGenFields} for a discussion.

This article is concerned with a question explicitly raised by Pop in \cite{elemEquivVsIsomI} (although implicitly asked earlier), namely whether non-isomorphic finitely generated fields are distinguishable in first-order logic, i.e.\ have different first-order theory.
This may be strengthened to the question whether for every finitely generated field $K$ there is a single sentence $\varphi_K$ in the language of rings such that for any finitely generated field $L$ we have $L \models \varphi_K$ if and only if $L \cong K$.
The analogous question for rings was recently answered affirmatively in \cite{AschenbrennerKhelifNaziazenoScanlon}.

Pop in \cite{elemEquivVsIsomII} answered this stronger question positively for finitely generated fields $K$ of Kronecker dimension $< 3$.
Recall here that the Kronecker dimension of a field $K$ is the transcendence degree of $K$ over its prime field if $K$ is of positive characteristic, and one plus the transcendence degree over $\Q$ if $K$ is of characteristic zero; a finitely generated field of Kronecker dimension $2$ is hence a function field in one variable over a global field.

We consider finitely generated $K$ of Kronecker dimension $d \geq 2$. We prove the following.
\begin{theorem}\label{thm:intro_exists_small_definable_subring}
  If $\kar K = 2$, assume that $d \leq 3$ or resolution of singularities holds over finite fields of characteristic two (Hypothesis \ref{hyp:resolution_of_singularities}).
  There is a subring of $K$, first-order definable with parameters, which is a finitely generated algebra over the prime field and has quotient field $K$.
\end{theorem}

Using results adapted from \cite{AschenbrennerKhelifNaziazenoScanlon}, one deduces a positive answer to the question above.
\begin{corollary}[Corollary \ref{cor:final_biinterpretable_quasiaxiomatisable}] \label{cor:intro_biinterpretable_quasiaxiomatisable}
   If $\kar K = 2$, assume that $d \leq 3$ or resolution of singularities holds over finite fields of characteristic two.
   Then $K$ as a structure in the language of rings is bi-interpretable with $\Z$.
   In particular, there is a sentence $\varphi_K$ such that for any finitely generated field $L$ we have $L \models \varphi_K$ if and only if $L \cong K$.
\end{corollary}
Here bi-interpretability is a somewhat subtle notion from model theory, see \cite[Section 2]{AschenbrennerKhelifNaziazenoScanlon}, which is a priori stronger than the axiomatisability property from Pop's question; however, all known approaches to the question do in fact yield bi-interpretability with $\Z$.

Bi-interpretability with $\Z$ implies in particular that ``any conceivable subset is definable'': one may label the elements of $K$ by the integers in such a way that addition and multiplication in $K$ are arithmetically definable, and any arithmetically definable subset of the integers is already definable in the field $K$ (\cite[Lemma 2.17]{AschenbrennerKhelifNaziazenoScanlon}).

The proof of Theorem \ref{thm:intro_exists_small_definable_subring} is based at its core on Pfister forms, which have frequently been used in definability problems over finitely generated fields, for instance in \cite{elemEquivVsIsomI} and \cite{poonenUniformDefnsInFinGenFields}.
We combine this with a local--global principle in cohomology, first conjectured by Kato, that has also been used in \cite{elemEquivVsIsomII}.
One innovation over \cite{elemEquivVsIsomII} lies in using this local--global principle without requiring resolution of singularities in characteristic away from two, by relying on alterations instead.

\subsection{Proof outline}

The method is in the wider sense a variation of the technique used in \cite{poonenUniversalExistential} for a definition of $\Z$ in $\Q$.
We define a predicate $S_c(\llangle a_0, \dotsc, a_d\rrsquare/K)$, associating to elements $c, a_0, \dotsc, a_d \in K$ a subset of $K$; this predicate is diophantine, i.e.\ existentially definable.

For certain values of $c$ and the $a_i$, we can determine the set $S_c(\llangle a_0, \dotsc, a_d\rrsquare/K)$ by using a local--global principle (Corollary \ref{cor:S_loc_glob}) to reduce to the simpler situation of an henselian field with finite residue field, and in this situation we have a good partial description (Proposition \ref{prop:local_S_computation}).

With some restriction on $c$ and the $a_i$, the set $S_c(\llangle a_0, \dotsc, a_d\rrsquare/K)$ is ``almost'' an intersection of valuation rings: There exists a set of valuation rings with finite residue field such that $S_c(\llangle a_0, \dotsc, a_d\rrsquare/K)$ is contained in their intersection, and contains the intersection of their maximal ideals.
This allows us to define a large family of subrings of $K$ in Proposition \ref{prop:constructing_ring}, by a simple argument using the constructible topology on the space of valuations on $K$.

Taking the intersection over a suitable subfamily, we obtain a subring of $K$ finitely generated over the prime field.

The treatment of a field of characteristic zero inductively relies on the case of positive odd characteristic. 
The case of characteristic two is not relied on by the other cases and requires some special care, mainly due to Galois cohomology with $\Z/2$-coefficients being unsatisfactory in this situation. The reader may hence choose to ignore this case throughout.

The use of the terminology and techniques from mathematical logic is confined to the last section, while the bulk of the work is algebraic.

\subsection{Acknowledgements}

The contents of sections \ref{sec:defn_of_S}, \ref{sec:over_henselian_fields} and \ref{sec:local_global} have previously appeared in the third chapter of my doctoral dissertation \cite{DittmannThesis}. I would like to take this opportunity to thank my doctoral adviser, Jochen Koenigsmann, for his support and advice.

\section{Pfister forms, and the definition of $S$}
\label{sec:defn_of_S}

The use of Pfister forms for definability problems over finitely generated fields is well-established, see for instance \cite{elemEquivVsIsomI} or \cite{poonenUniformDefnsInFinGenFields}; however, characteristic two seems to have been avoided so far.
We follow the terminology of \cite[Chapter II]{ElmanKarpenkoMerkurjev_AlgGeomTheoryOfQuadForms} in all characteristics.

Over a field $K$ of characteristic not two, a $1$-fold Pfister form is a quadratic form $(x,y) \mapsto x^2 - ay^2$ for some $a \in K^\times$.
On the other hand, for $K$ of characteristic two a $1$-fold Pfister form is of the form $(x,y) \mapsto x^2 + xy + ay^2$ for some $a \in K$.
In either case, this $1$-fold Pfister form is denoted $\llangle a\rrsquare$.

Inductively, we define a $k+1$-fold Pfister form $\llangle a_1, \dotsc, a_{k+1}\rrsquare$ to be the orthogonal sum $\llangle a_2, \dotsc, a_{k+1}\rrsquare \bot (-a_1)\llangle a_2, \dotsc, a_{k+1}\rrsquare$.
In this way, we have $k$-fold Pfister forms for all $k \geq 1$, and these are quadratic forms of dimension $2^k$.

We call the Pfister form $q$ \emph{isotropic} (over $K$) if it has a non-trivial zero in $K^{2^k}$, and \emph{anisotropic} otherwise.

\begin{definition}\label{defn:S}
  Given a $k$-fold Pfister form $q = \llangle a_1, \dotsc, a_k\rrsquare$ over $K$ and an element $c \in K$, we define $S_c(q/K) \subseteq K$ as follows.
  \begin{itemize}
  \item If $q$ is isotropic over $K$, we let $S_c(q/K) = K$.
  \item If $q$ is anisotropic over $K$, we let
    \begin{multline*} S_c(q/K) = \{ x \in K \colon X^2+ (1-x)X+c \text{ is irreducible over $K$ and} \\ \text{$q$ is isotropic over $K[X]/(X^2+(1-x)X+c)$} \}. \end{multline*}
  \end{itemize}
\end{definition}
For an extension field $L/K$, we can interpret $q$ as a Pfister form over $L$ (i.e.~we notationally suppress the base change of quadratic forms), given by the same $a_i$, and refer to $S_c(q/L)$.
Evidently we have $S_c(q/L) \supseteq S_c(q/K)$ in this situation.
\begin{remark}\label{rem:alternative_defn_S'}
  The definition of $S_c(q/K)$ may look very ad hoc, but can equivalently be phrased without a case distinction. Fix $x \in K$ and write $A=K[X]/(X^2+(1-x)X+c)$. Then the following are equivalent.
  \begin{itemize}
    \item $x \in S_c(q/K)$;
    \item $q$ has a zero in $\mathbb{P}^{2^k-1}_A$;
    \item there exists a zero $x_1, \dotsc, x_{2^k} \in A$ of $q$ such that $(x_1, \dotsc, x_{2^k})$ is the unit ideal in $A$.
    \end{itemize}
    The equivalence is clear if $X^2+(1-x)+c$ is irreducible over $K$ and so $A$ is a field; otherwise, $X^2+(1-x)X+c$ has a linear factor, and hence we have $K$-algebra homomorphisms $K \hookrightarrow A \twoheadrightarrow K$, so all three conditions are equivalent to $q$ being isotropic over $K$.
\end{remark}

\begin{remark}
  In \cite[Section 2]{poonenUniversalExistential}, Poonen defines a set $S_{a,b} = \{ 2x \in K \colon \exists y,z,w \in K \colon x^2 - ay^2 - bz^2 + abw^2 = 1 \}$, working in characteristic away from two.
  One can show, see \cite[Proposition 2.2.3]{DittmannThesis}, that this relates to the definition above by $S_{a,b} = (1 - S_1(\llangle a, b \rrsquare/K)) \cup \{ -2, 2\}$.
  This may serve as motivation for our definition, but will play no role in the sequel.
\end{remark}

To investigate the sets $S_c(q/K)$, we use a connection with Galois cohomology.
For a field $K$ and integer $i \geq 1$, we write $H^{i}(K)$ for the Galois cohomology group $H^i(K, \Z/2)$ if $\kar K \neq 2$.
If $\kar K = 2$, we write $H^i(K)$ for the Galois cohomology group $H^1(K, K_{i-1}^M(K^{\mathrm{sep}})/2)$, where  $K_{i-1}^M(K^{\mathrm{sep}})$ is the $(i-1)$-th Milnor $K$-group of a separable closure of $K$. 
This definition follows \cite[§101]{ElmanKarpenkoMerkurjev_AlgGeomTheoryOfQuadForms}, where it is also notated $H^{i, {i-1}}(K, \Z/2)$, and agrees with the group $H^i(K, \Z/2(i-1))$ in \cite{kato} (in characteristic two, this uses the Bloch--Gabber--Kato theorem on the bijectivity of the differential symbol \cite[Theorem 9.5.2]{centralSimpleAlgebrasAndGalCohom}).

To a $k$-fold Pfister form $q = \llangle a_1, \dotsc, a_k\rrsquare$ we associate a cohomology class in $H^k(K)$ in the following way, following \cite[§16]{ElmanKarpenkoMerkurjev_AlgGeomTheoryOfQuadForms}:
In characteristic not two, we may associate to each $a_i$ its square class in $K^\times / 2$, and this group is isomorphic to $H^1(K,\Z/2)$ by the Kummer isomorphism. (Note that the Galois module $\Z/2$ is isomorphic to $\mu_2$.) To the Pfister form $q$, we associate the cup product $\alpha = (a_1) \cup \dotsb \cup (a_k)$, where $(a_i)$ is the element of $H^1(K,\Z/2)$ corresponding to $a_i$.

In characteristic two, to each $a_i$ with $i < k$ we associate its square class in $K^\times/2 \cong K_1^M(K)/2 \to H^0(K, K_1^M(K^{\mathrm{sep}})/2)$, and to $a_k$ we associate the element of $H^1(K) = H^1(K, \Z/2)$ given by the Artin-Schreier isomorphism $K / \wp(K) \cong H^1(K, \Z/2)$, see \cite[Corollary 6.1.2]{NeukirchSchmidtWingberg}. Again the cup product produces an element $\alpha \in H^k(K)$.

\begin{fact}\label{fact:pfister_forms_and_cohomology}
  Let $q = \llangle a_1, \dotsc, a_k\rrsquare$ be a $k$-fold Pfister form over a field $K$ of arbitrary characteristic and $\alpha \in H^k(K)$ its associated cohomology class.
  Then $q$ is isotropic over $K$ if and only if $\alpha$ vanishes.
\end{fact}

This is stated in \cite[Fact 16.2]{ElmanKarpenkoMerkurjev_AlgGeomTheoryOfQuadForms}. In characteristic not two, it follows from the Milnor Conjectures on the graded ring of Milnor-$K$-Theory mod $2$, the graded Witt ring, and the graded $\Z/2$-cohomology ring being isomorphic; see \cite[§16]{ElmanKarpenkoMerkurjev_AlgGeomTheoryOfQuadForms} for a discussion.

As an immediate consequence, for any overfield $L/K$ the form $q$ becomes isotropic over $L$ if and only if $\alpha$ is annihilated by the restriction map $H^k(K) \to H^k(L)$.

\section{Over henselian fields}
\label{sec:over_henselian_fields}

In this section we work with valued fields $(F,v)$. We write $\mathcal{O}_v$ for the valuation ring, $\mathfrak{m}_v$ for its maximal ideal, $Fv$ for the residue field, $vF$ for the value group (always written additively), and $F_v$ for an henselisation.
Throughout, we exclude the case of mixed characteristic $(0,2)$, i.e.~we assume that either $\kar(Fv) \neq 2$ or $\kar F = \kar(Fv) = 2$.
The main objective of this section is to establish the following.

\begin{proposition}\label{prop:local_S_computation}
  Let $(F,v)$ be henselian with finite residue field and value group a lexicographic power $\Z^r$, and $c \in \mathcal{O}_v$ such that the reduction of $X^2+X+c$ is irreducible over the residue field.
  If $\kar Fv = 2$, then assume that $(F,v)$ is a henselisation of a finitely generated extension of $\F_2$ of transcendence degree $r$.
  
  Then for any anisotropic $(r+1)$-fold Pfister form $q/F$ we have \[ \mathfrak{m}_v \subseteq S_c(q/F) \subseteq \mathcal{O}_v .\]
\end{proposition}

Our main tools for the proof are cohomological.
\begin{proposition}\label{prop:local_cohomology}
  Let $(F,v)$ be henselian of residue characteristic not two with $vF \cong \Z$.
  \begin{enumerate}
  \item There is a family of surjective homomorphisms $H^{m+1}(F) \to H^m(Fv)$; write $\partial_v$ for all of them.
  \item When $m = \cd_2(Fv)$, $\partial_v$ is an isomorphism. (Here $\cd_2$ stands for the \emph{$2$-cohomological dimension}.)
  \item For any finite extension $E/F$, the following diagram commutes.
    \[ \xymatrix{ H^{m+1}(F) \ar[r] \ar[d] & H^m(Fv) \ar[d]^{\cdot e} \\ H^{m+1}(E) \ar[r] & H^m(Ev) } \]
    Here the horizontal maps are given by $\partial_v$, the vertical map on the left-hand side is cohomological restriction, and the vertical map on the right-hand side is restriction followed by multiplication by the ramification index $e = (vE \colon vF)$.
  \item Let $x \in \mathcal{O}_v^\times$, so $x$ induces an element $\alpha = (x) \in H^1(F)$ and an element $\overline\alpha = (\overline x) \in H^1(Fv)$ via the Kummer map.
    Then for any $m$ the following diagram commutes, where the vertical maps are given by the cup-product with $\alpha$ and $\overline\alpha$.
    \[\xymatrix{
        H^{m+1}(F) \ar[r]^{\partial_v} \ar[d]^{\alpha \cup \cdot} & H^m(Fv) \ar[d]^{\overline\alpha \cup \cdot} \\
        H^{m+2}(F) \ar[r]^{\partial_v} & H^{m+1}(Fv)
      }\]
  \end{enumerate}

  If $(F,v)$ is henselian with $\kar F = \kar Fv = 2$ and $vF \cong \Z$, then there also exist homomorphisms $H^{m+1}(F) \to H^m(Fv)$ for those $m$ with $[Fv : (Fv)^2] < 2^m$, also denoted by $\partial_v$.
  These make the diagram (3) commute.

  If furthermore the valuation ring of $v$ is the henselisation of an excellent valuation ring, or an unramified extension thereof, then the homomorphisms $\partial_v$ are isomorphisms whenever defined.
\end{proposition}
The additional requirement that the valuation ring be excellent in characteristic two forces us to make the additional assumptions in Proposition \ref{prop:local_S_computation}.
\begin{proof}
  In the case of $\kar Fv \neq 2$, the maps are constructed in \cite[§1]{kato}, but it is useful to have more explicit descriptions, which can be found in many places in the literature.
  An explicit construction of the maps $\partial_v$ is for instance given in \cite[Construction 6.8.5]{centralSimpleAlgebrasAndGalCohom} (there only for complete discretely valued fields, but inspection of the proofs shows that henselianity is sufficient).
  This construction agrees with the one given in \cite{kato} by \cite[Proposition 6.8.2, Remark 6.8.3]{centralSimpleAlgebrasAndGalCohom}.
  
  By \cite[Corollary 6.8.8]{centralSimpleAlgebrasAndGalCohom}, the kernel of $\partial_v \colon H^{m+1}(F) \to H^m(Fv)$ is isomorphic to $H^{m+1}(Fv)$, which is trivial if $\cd_2(Fv) \leq m$.
  Compatibility with finite field extensions (3) can be read off from the construction, but also follows from \cite[Remark 7.1.6(2), Proposition 7.5.1]{centralSimpleAlgebrasAndGalCohom}.
  The commutative diagram from (4) follows from \cite[Lemma 6.8.4]{centralSimpleAlgebrasAndGalCohom}.

  For $\kar F = \kar Fv = 2$, the construction of the maps $\partial_v$ is again given in \cite[§1]{kato}, where the hypothesis $[Fv : (Fv)^2] < 2^m$ is needed.
  The commutativity of diagram (3) can be read off from the construction.
  It is given in \cite[Lemma 1.4(3)]{kato} that the maps $\partial_v$ are isomorphisms when defined under the additional hypothesis that $\mathcal{O}_v$ is the henselisation of an excellent ring.
  The same applies to finite unramified extensions of such rings, since passing to a finite extension preserves excellency. By the commutative diagram (3), we may pass to arbitrary unramified extensions.
\end{proof}

\begin{remark}
  In the case $\kar K \neq 2$, it is in fact unnecessary to restrict to value groups isomorphic to $\Z$; for any value group $vF$ with $r = \dim_{\F_2}vF/2vF$ finite one can construct canonical residue maps $H^{m+r}(F,\Z/2) \to H^m(Fv,\Z/2)$ satisfying all of the properties above.
  This may be deduced from \cite[Theorem 3.6, Remark 3.12]{WadsworthPHenselianFieldsKTheoryGalCohomGWittRings}.
\end{remark}

\begin{lemma}\label{lem:excellent_dvr}
  Let $(F,v)$ be a valued field with $F$ finitely generated of transcendence degree $r$ over $\F_2$ and $Fv$ of transcendence degree $r-1$ over $\F_2$. Then the valuation ring of $v$ is excellent.
\end{lemma}
\begin{proof}
  The field $F$ is the function field of some integral projective variety over $\F_2$, see \cite[Proposition 2.2.13]{poonenRationalPoints}. By \cite[Theorem 3.26(b)]{LiuAlgGeomAndArithCurves}, after replacing the variety by a birational one if necessary, the valuation ring of $v$ is the local ring of a point of codimension $1$ on the variety, and hence excellent by \cite[Corollary 2.40]{LiuAlgGeomAndArithCurves}.
\end{proof}

\begin{lemma}\label{lem:isotropic_over_unramified_extn}
  Let $(F,v)$ be henselian with finite residue field and value group a lexicographic power $\Z^r$, and $E/F$ the unique unramified quadratic extension.
  If $\kar Fv = 2$, assume that $(F,v)$ is a henselisation of a finitely generated extension of $\F_2$ of transcendence degree $r$.
  Then every $(r+1)$-fold Pfister form over $F$ becomes isotropic over $E$.
\end{lemma}
\begin{proof}
  Since $v$ has value group $\Z^r$, we may write $v$ as an iterated composition of valuations $v_1 \circ \dotsb \circ v_r$ (in the sense of \cite[p.~45]{englerPrestel}), where each $v_i$ has value group $\Z$; this means that we have valued fields $(F_i, v_i)$ with $F_r = F$, $F_i = F_{i+1} v_{i+1}$ for $i < r$, and $F_1v_1 = Fv$.
  We shall argue that the conditions from Proposition \ref{prop:local_cohomology} are satisfied.
  
  If $\kar Fv \neq 2$, then since the residue field $Fv = F_1v_1$ is finite and hence of $2$-cohomological dimension $1$, and for each $i$ we have $\cd_2(F_i) = \cd_2(F_iv_i) + 1$ by \cite[II.4.3, Proposition 12]{cohomologieGaloisienne} (there stated for complete valued fields, but inspection of the proof shows that henselianity suffices), we have $\cd_2(F_iv_i) = i$ and $\cd_2(F) = r+1$.
  If $\kar F = \kar Fv = 2$, then because each $F_i$ has strictly higher transcendence degree over the prime field than $F_i v_i$, we see that each $F_i$ must have transcendence degree precisely $i$ over $\F_2$, and in particular $[F_i : F_i^2] \leq 2^i$. 
  By the condition that $(F,v)$ is a henselisation of a finitely generated extension of $\F_2$ of transcendence degree $r$, we see by \cite[Corollary 4.1.4]{englerPrestel} that each $(F_i, v_i)$ is in fact an unramified extension of a henselisation of a discretely valued field to which Lemma \ref{lem:excellent_dvr} is applicable.

  Irrespective of the characteristics, at each step we therefore have the isomorphism $\partial_{v_i} \colon H^{i+1}(F_i) \to H^i(F_iv_i)$ from above.
  Composition gives an isomorphism $\partial_v \colon H^{r+1}(F) \to H^1(Fv)$. The same construction applies to the unramified extension $E/F$.
  The diagram \[ \xymatrix{ H^{r+1}(F) \ar[r] \ar[d] & H^1(Fv) \ar[d] \\ H^{r+1}(E) \ar[r] & H^1(Ev) } \]
  commutes by Proposition \ref{prop:local_cohomology}(3), where the horizontal maps are given by $\partial_v$ and the vertical maps are restrictions.
  Since the restriction map $H^1(Fv) \to H^1(Ev)$ is the zero map as $Ev$ is the only quadratic extension of $Fv$, we deduce that the restriction map $H^{r+1}(F) \to H^{r+1}(E)$ is also the zero map.
  Hence every $(r+1)$-fold Pfister form over $F$ becomes isotropic over $E$.
\end{proof}

\begin{lemma}\label{lem:henselian_S_in_valn_ring}
  Let $(F,v)$ be a henselian field.
  If $q/F$ is an anisotropic Pfister form and $c \in \mathcal{O}_v$, then $S_c(q/F) \subseteq \mathcal{O}_v$.
\end{lemma}
\begin{proof}
  If $x \in F$ with $vx < 0$, then the polynomial $X^2+(1-x)X+c$ is reducible in $F$ by Hensel's Lemma, and hence $x \not\in S_c(q/F)$.
\end{proof}

\begin{proof}[Proof of Proposition \ref{prop:local_S_computation}]
  The inclusion $S_c(q/F) \subseteq \mathcal{O}_v$ is immediate from Lemma \ref{lem:henselian_S_in_valn_ring}.

  For the other inclusion, let $x \in \mathfrak{m}_v$. The polynomial $X^2+(1-x)X+c$ reduces to an irreducible polynomial over $Fv$ by assumption, and therefore $E:= F[X]/(X^2+(1-x)X+c)$ is an unramified extension field of $F$.
  By Lemma \ref{lem:isotropic_over_unramified_extn}, the Pfister form $q$ becomes isotropic over $E$, and hence $x \in S_c(q/F)$.
\end{proof}

For later use we give a few easy facts.
\begin{lemma}\label{lem:anisotropic_over_dvf}
  Let $(F,v)$ be a valued field with $vF \cong \Z$ and uniformiser $\pi$. If $\kar Fv = 2$, assume that $\kar F = 2$.
  Let $a_1, \dotsc, a_k \in \mathcal{O}_v^\times$ such that the Pfister form $\llangle \overline a_1, \dotsc, \overline a_k\rrsquare/Fv$ given by the residues $\overline a_i$ of the $a_i$ under the valuation $v$ is anisotropic.
  Then $\llangle \pi, a_1, \dotsc, a_k\rrsquare$ is anisotropic over $F$.
\end{lemma}
\begin{proof}
  Write $q = \llangle a_1, \dotsc, a_k\rrsquare$. If $\llangle \pi, a_1, \dotsc, a_k\rrsquare = q \bot (-\pi)q$ is isotropic over $F$, then there exist tuples $x, y \in F^{2^k}$ with $q(x) = \pi q(y)$.
  If all entries of $x$ are in $\mathcal{O}_v$, but not all in the maximal ideal, then $q(x)$ has valuation zero by anisotropy of the residue form.
  By scaling if necessary, we see that $q(x)$ always has even valuation unless $x$ is the zero vector, and by the same argument $\pi q(y)$ always has odd valuation unless $y$ is the zero vector.
  Hence the equality $q(x) = \pi q(y)$ has no non-trivial solution.
\end{proof}

\begin{lemma}\label{lem:isotropic_in_henselisation_char_not_two}
  Let $(F,v)$ be a valued field of residue characteristic not two and $q/F$ a $k$-fold Pfister form presented as $\sum_{i=1}^{2^k} b_i X_i^2$ with $b_i \in F^\times$.
  Then $q$ is isotropic over the henselisation $F_v$ if and only if there exist $x_1, \dotsc, x_{2^k} \in F$ not all zero with $v(q(x_1, \dotsc,x_{2^k})) > \min_i v(b_ix_i^2)$.
\end{lemma}
\begin{proof}
  Assume first that $q$ is isotropic over the henselisation $F_v$, so there are $x_1, \dotsc, x_{2^k} \in F_v$, not all zero, with $q(x_1, \dotsc, x_{2^k}) = 0$.
  For each $i$, choose $x_i' \in F$ by setting $x_i' = 0$ if $x_i = 0$, and ensuring $v(x_i/x_i' -1) > 0$ otherwise; this is possible since $F_v$ has the same value group and residue field as $F$. This gives $v(x_i^2-x_i'^2) > v(x_i'^2)$ if $x_i \neq 0$.
  Then \[v(q(x_1', \dotsc, x_{2^k}')) = v(\sum_i b_i(x_i'^2-x_i^2))> \min_i v(b_i x_i'^2) \] as desired.

  For the converse direction, let $x_1, \dotsc, x_{2^k} \in F$ not all zero with $v(q(x_1, \dotsc, x_{2^k})) > \min_i v(b_ix_i^2)$.
  Assume the minimum on the right-hand side is attained at index $j$.
  Set $y = -(b_j x_j^2)^{-1} \sum_{i \neq j} b_i x_i^2$.
  We obtain $v(y-1) > 0$, so $y$ is a square in $F_v$ by Hensel's Lemma, say $y = z^2$.
  Then $\sum_{i \neq j} b_ix_i^2 + b_j (x_jz)^2 = 0$, so $q$ is isotropic over $F_v$.
\end{proof}

\begin{lemma}\label{lem:isotropic_in_henselisation_char_two}
  Let $(F,v)$ be a valued with $\kar F = 2$ and $q = \llangle a_1, \dotsc, a_{k-1}, a\rrsquare$ a $k$-fold Pfister form over $F$ presented as $\sum_{i=1}^{2^{k-1}} b_i (X_i^2 + X_i Y_i + a Y_i^2)$.
  Assume that $v(a) \geq 0$.
  Then $q$ is isotropic over the henselisation $F_v$ if and only if there exist $x_1, \dotsc, x_{2^{k-1}}, y_1, \dotsc, y_{2^{k-1}} \in F$ not all zero with $v(q(x_1, \dotsc, x_{2^{k-1}}, y_1, \dotsc, y_{2^{k-1}})) > \min \{ v(b_i x_i^2), v(b_i x_i y_i), v(b_i a y_i^2) \colon i \geq 1 \}$.
\end{lemma}
\begin{proof}
  We work as in the preceding lemma. If $q$ is isotropic over the henselisation $F_v$, so there is a non-trivial zero $x_1, \dotsc, x_{2^{k-1}}, y_1, \dotsc, y_{2^{k-1}}$ of $q$, we pick $x_i'$ and $y_i'$ such that $x_i'= 0$ if $x_i=0$ and $v(x_i/x_i'- 1) > 0$ otherwise, and similarly for $y_i'$.
  Then it is easy to verify the inequality $v(q(x_1', \dotsc, x_{2^{k-1}}', y_1', \dotsc, y_{2^{k-1}}')) > \min \{ v(b_i x_i'^2), v(b_i x_i' y_i'), v(b_i a y_i'^2) \colon i \geq 1 \}$.
  
  Assume conversely that $x_i, y_i$ are given such that the inequality from the statement is satisfied.
  Observe first that if $v(a) > 0$, then $q$ has a non-trivial zero in the henselisation $F_v$ since $X^2 + X + a$ has a zero by Hensel's Lemma, so let us assume that $v(a) = 0$. In particular, we have $v(b_i x_i y_i) \geq \min \{ v(b_i x_i^2), v(b_i a y_i^2))\}$ for every $i$.
  Let us assume that the minimum in $\min \{ v(b_i x_i^2), v(b_i x_i y_i), v(b_i a y_i^2) \colon i \geq 1 \}$ is attained in the term $v(b_j a y_j^2)$; the argument will work analogously if it is attained in $v(b_j x_j^2)$.
  By scaling, we may assume that $y_j = 1$.
  Hensel's Lemma allows us to find $z \in F_v$ with
  $z^2 + z = q(x_1, \dotsc, x_{2^{k-1}}, y_1, \dotsc, y_{2^{k-1}})/b_j$.
  Since the form $b_j (X_j^2 + X_j Y_j + a Y_j^2)$ is additive in both variables, we obtain an exact zero of $q$ in $F_v$ by replacing $x_j$ by $x_j + z$.
\end{proof}

We also state some results for $2$-fold Pfister forms over global fields for later use.
\begin{lemma}\label{lem:global_fields_local_global}
  A $2$-fold Pfister form over a global field is isotropic if becomes isotropic over all non-trivial henselisations and real completions.
\end{lemma}
Observe here that a Pfister form over a global field becomes isotropic over some henselisation if and only if it does so over the corresponding completion, as is easy to see by the density of global fields in their completions and Hensel's Lemma.
\begin{proof}
  In characteristic away from two, this follows immediately from the Hasse--Minkowski local--global principle on quadratic forms. 
  In full generality, one uses that for any field $K$, $H^2(K)$ embeds canonically into the Brauer group by \cite[Example 101.1(4)]{ElmanKarpenkoMerkurjev_AlgGeomTheoryOfQuadForms} (this effectively associates a quaternion algebra with its reduced norm form, a $2$-fold Pfister form) and then applies the Albert--Brauer--Brauer--Hasse--Noether local--global principle for the Brauer group.
\end{proof}

\begin{lemma}\label{lem:pfister_over_global_fields}
  Let $K_1$ be a global field and $\llangle a_0, a_1\rrsquare/K_1$ a $2$-fold Pfister form. For almost all valuations $v$ on $K_1$, the Pfister form $\llangle a_0, a_1\rrsquare$ becomes isotropic over the henselisation $(K_1)_v$.
  
  Furthermore, for any valuation $v$ on $K_1$ there exists a Pfister form $\llangle a_0, a_1\rrsquare/K_1$ which does not become isotropic over the henselisation $(K_1)_v$.
  If $\kar K_1 = 0$ and $v$ is not of residue characteristic $2$, we may additionally ensure that $\llangle a_0, a_1\rrsquare$ does become isotropic over every real completion and every completion with residue characteristic two.
\end{lemma}
\begin{proof}
  All of this and more follows easily from the determination of $H^2(K_1)$ in class field theory, see \cite[Theorem 8.1.16]{NeukirchSchmidtWingberg}, but we give an ad hoc argument here.

   Given a Pfister form $\llangle a_0, a_1\rrsquare/K_1$, for any non-trivial valuation $v$ not of mixed characteristic $(0,2)$ with $v(a_0)=v(a_1)=0$, the residue Pfister form $\llangle \overline a_0, \overline a_1\rrsquare/K_1v$ has a zero over the $C_1$ field $K_1v$ (\cite[Theorem 6.2.6]{centralSimpleAlgebrasAndGalCohom}), and this gives a zero of $\llangle a_0, a_1\rrsquare/(K_1)_v$ by Lemmas \ref{lem:isotropic_in_henselisation_char_not_two} and \ref{lem:isotropic_in_henselisation_char_two}.
   Since almost all valuations on $K_1$ satisfy $v(a_0)=v(a_1)=0$, this proves the first part.

  For the second part, we may assume that $v$ is non-trivial. Choose $a_0, a_1 \in K_1$ such that $a_0$ is a uniformiser for $v$ and $v(a_1)=0$ with $\llangle \overline a_1 \rrsquare$ anisotropic over $K_0 v$, which is possible since $K_1 v$ has a separable quadratic extension; then $\llangle a_0, a_1\rrsquare$ is anisotropic over $(K_1)_v$ by Lemma \ref{lem:anisotropic_over_dvf}.
  In characteristic zero we can use Weak Approximation \cite[Theorem II.3.4]{Neukirch} to furthermore force $a_0 > 0$ for all orderings $>$ of $K_0$, and $w(a_0-1) > w(8)$ for all places $w$ of residue characteristic two, which means that $a_0$ is a square at all real completions and all completions of residue characteristic two, so $\llangle a_0, a_1\rrsquare$ becomes isotropic over those completions.
\end{proof}

\begin{lemma}\label{lem:pfister_over_global_fields_special_form}
  Every $2$-fold Pfister form $q$ over a global field $K_1$ is isomorphic to a form $\llangle b_0, b_1\rrsquare$ such that $v(b_1) \geq 0$ for all valuations $v$ on $K_1$ such that $q$ does not become isotropic over the henselisation $(K_1)_v$.
\end{lemma}
\begin{proof}
  Let $q$ be given as $\llangle a_0, a_1\rrsquare$.
  Note first that in characteristic away from two we may multiply $a_1$ by an arbitrary square without affecting the isomorphism type of the Pfister form, which proves the claim, so let us assume that $\kar K_0 = 2$. We may assume that $q$ is anisotropic.
  Use weak approximation to choose $b_1 \in K_1$ such that for each valuation $v$ on $K_1$ over whose henselisation $(K_1)_v$ the form $q$ does not split we have $v(b_1) \geq 0$ and $X^2 + X + b_1$ irreducible over the residue field.
  Then all valuations $v$ in question are inert in the field extension $L=K_1[X]/(X^2+X+b_1)$ of $K$ and hence $A$ becomes split over $L$ by Lemma \ref{lem:global_fields_local_global}.
  It follows by \cite[Theorem 34.22(3)]{ElmanKarpenkoMerkurjev_AlgGeomTheoryOfQuadForms} that $q$ is isomorphic to $\llangle b_0, b_1\rrsquare$ for some $b_0 \in K_1^\times$.
\end{proof}

\section{The local--global principle over finitely generated fields}
\label{sec:local_global}

Let $K$ be a finitely generated field of Kronecker dimension $d \geq 1$; recall that this means that $K$ has transcendence degree $d-1$ over a global field.
If $K$ has characteristic two, we will make use of the following hypothesis.
\begin{hypothesis}\label{hyp:resolution_of_singularities}
Following \cite[Definition 4.18]{JannsenHassePrinciples}, we say that \emph{resolution of singularities holds over finite fields of characteristic two}, if for every integral and proper variety over a finite field $\F$ of characteristic two there exists a proper birational morphism from a smooth variety over $\F$, and furthermore any smooth affine variety over such a finite field $\F$ can be realised as an open subvariety of a projective smooth variety over $\F$ whose complement is a simple normal crossings divisor.
\end{hypothesis}

We can now state an important result on $(d+1)$-fold Pfister forms over $K$, relating their isotropy to the isotropy over certain henselisations.
Note here that it follows from standard results in valuation theory, see \cite[Theorem 3.4.3]{englerPrestel} and its corollaries, that any valuation $v$ on $K$ has value group of archimedean rank $\leq d$, and that if it has archimedean rank $d$, its value group is isomorphic to the lexicographic power $\Z^d$ and its residue field is finite.
\begin{theorem}\label{thm:pfister_loc_glob}
  Let $q = \llangle a_0, \dotsc, a_d \rrsquare$ be a $(d+1)$-fold Pfister form over $K$.
  If $\kar K = 0$, assume that $a_{d-1}, a_d \in \Q$ and the Pfister form $\llangle a_{d-1}, a_d\rrsquare/\Q$ becomes isotropic over $\Q_2$ and $\R$.
  If $\kar K = 2$, assume that $d \leq 3$ or resolution of singularities holds over finite fields of characteristic two.

  Then for any finite extension $L/K$ such that $q$ is anisotropic over $L$, there exists a valuation $v$ of archimedean rank $d$ such that $q$ remains anisotropic over the henselisation $L_v$. Furthermore we can take $v$ to be of residue characteristic not two if $\kar K = 0$.
\end{theorem}
Note that when $d=1$, i.e.\ $K$ is a global field, the statement follows from the Hasse--Minkowski Theorem on quadratic forms. One may also compare with the Hasse--Brauer--Noether Theorem in the case of a quaternion algebra.
The special condition in characteristic zero of $\llangle a_0, a_1\rrsquare$ becoming isotropic over $\Q_2$ and $\R$ is technical; it will reappear in subsequent lemmas.

Theorem \ref{thm:pfister_loc_glob} will be used in the form of the following corollary, the main result of this section.
\begin{corollary}\label{cor:S_loc_glob}
  For $q$ as above, and under the same assumptions as above if $\kar K = 2$, we have
  \[S_c(q/K) = \bigcap_v S_c(q/K_v) \cap K, \]
  where $v$ varies over valuations on $K$ of archimedean rank $d$. We may restrict to those $v$ of residue characteristic not two if $\kar K = 0$.
\end{corollary}
\begin{proof}
  If $q$ is isotropic over $K$, then both sides are equal to $K$, so assume this is not the case.
  The inclusion $\subseteq$ is immediate from the definition of $S_c$.

  For the other inclusion, let $x \in K$ not be contained in the left-hand side.
  Suppose first that the polynomial $X^2+(1-x)X+c$ is irreducible over $K$. Then $q$ is not isotropic over the field $L=K[X]/(X^2+(1-x)X+c)$, so by the theorem there is a rank-$d$ valuation $v$ on $L$ such that $q$ is anisotropic over $L_v$.
  Write $K_v$ for the henselisation of $K$ with respect to the restriction of $v$.
  If $X^2+(1-x)X+c$ is irreducible over $K_v$, then $L_v = K_v[X]/(X^2+(1-x)X+c)$, from which we deduce that $x \not\in S_c(q/K_v)$.
  Otherwise $X^2+(1-x)X+c$ is reducible over $K_v$ and $q$ is anisotropic over $K_v = L_v$, so likewise $x \not\in S_c(q/K_v)$.

  If $X^2+(1-x)X+c$ is reducible over $K$, we may simply pick any rank-$d$ valuation $v$ on $K$ such that $q$ is anisotropic over $K_v$; such a valuation exists by the theorem.
  Now $x \not\in S_c(q/K_v)$.
\end{proof}

To prove Theorem \ref{thm:pfister_loc_glob}, we reduce to a local--global principle in cohomology.
\begin{theorem}\label{thm:cohom_loc_glob}
  Let $K$ be as above and $\alpha \in H^{d+1}(K)$ be non-zero.
  If $\kar K = 0$, assume furthermore that $\alpha$ becomes trivial when restricted to any overfield of $K$ embedding $\R$ or $\Q_2$.
  If $\kar K = 2$, assume that $d \leq 3$ or resolution of singularities holds over finite fields of characteristic two.

  Then there exists a valuation $v$ on $K$, with value group $\Z$ and residue field $Kv$ finitely generated of Kronecker dimension $d-1$, with $\kar Kv \neq 2$ unless $\kar K = 2$, such that $\alpha$ is not annihilated by the composite map
  \[ H^{d+1}(K) \to H^{d+1}(K_v) \to H^d(Kv), \]
  where the first map is cohomological restriction to the henselisation $K_v$, and the second map is $\partial_v$ as in Proposition \ref{prop:local_cohomology}.
\end{theorem}

To prove this theorem in turn, we first reduce to the situation where $K$ has a regular proper model over $\Z[1/2]$ or a finite field $\mathbb{F}_p$, i.e.\ a regular integral variety flat and proper over these rings whose function field is $K$.

\begin{lemma}\label{lem:regular_model_after_odd_extn}
  For every finitely generated field $K$ -- assuming that $\kar \neq 2$, or $K$ is of transcendence degree at most $3$ over $\F_2$, or resolution of singularities holds over finite fields of characteristic two -- there exists a finite extension $L/K$ of odd degree and a regular integral variety $X$ proper and flat over $\Z[1/2]$ or a finite field $\F_p$ with function field $L$. In characteristic two, we may choose $L=K$.
\end{lemma}
In characteristic away from two this is a weak variant of resolution of singularities, which we deduce from results on alterations due to Gabber.
\begin{proof}
  Consider first the case of characteristic zero.
  Then $K$ has a proper model over $\Z[1/2]$ (or even over $\Z$), i.e.\ an integral scheme $V$ proper and flat over $\Z[1/2]$ such that $K$ is the residue field of the generic point of $V$.
  (Such a model is easily obtained by writing $K$ as $\Q(t_2, \dotsc, t_d, s)$ where the $t_i$ form a transcendence basis and $s$ is integral over $\Z[1/2][t_2, \dotsc, t_d]$, and taking $V$ to be the projective closure of $\operatorname{Spec}(\Z[1/2][t_2, \dotsc, t_d,s])$.)
  
  By \cite[Theorem 2.4]{travauxGabberX}, there exists a projective $2'$-alteration $X \to V$, i.e.\ a proper surjective generically finite maximally dominating (i.e.\ every irreducible component dominates $V$) morphism of odd degree at generic points, with $X$ regular.
  By replacing $X$ with one of its connected components we may take $X$ to be integral, since non-empty connected regular schemes are integral.
  The map $X \to V$ remains surjective, since it is still dominant and proper.
  
  The morphism $X \to V \to \operatorname{Spec}(\Z[1/2])$ is flat because it is a surjective morphism from an integral scheme to a Dedekind scheme, see \cite[Proposition 4.3.9]{LiuAlgGeomAndArithCurves}.
  Now the function field of $X$ is an odd degree extension of $K$, proving the claim.

  If $K$ has characteristic $p>0$, it has a proper model $V$ over $\F_p$ (by the same argument as above, see also \cite[Proposition 2.2.13]{poonenRationalPoints}). 
  If $p=2$, there exists a proper birational morphism from a smooth variety $X$ over $\F_2$, either by the assumption of resolution of singularities or by $d \leq 3$ since resolution is known up to dimension $3$ over perfect fields by \cite{CossartPiltant_ResolutionOfSingularitiesOfThreefoldsII}, so the function field of $X$ is $K$. Hence let us assume that $p>2$.
  By \cite[Theorem 2.1]{travauxGabberX}, there exists a projective $2'$-alteration $X \to V$, with $X$ smooth over a finite extension of $\F_p$ and hence regular.
  We proceed as above.
\end{proof}

In the situation where $K$ has a regular proper model, our cohomological local-global principle is related to local-global principles conjectured by Kato, given in the following two theorems.
\begin{theorem}[Kerz--Saito/Jannsen/Suwa]\label{thm:KerzSaitoFiniteField}
  Let $X$ be an integral variety proper and smooth of dimension $d$ over a finite field $\F$. If $\kar \F=2$ then assume that $d \leq 3$ or resolution of singularities holds over finite fields of characteristic two. Then
  \[ H^{d+1}(\F(X)) \to \bigoplus_x H^d(x) \]
  is injective.
  Here the sum is over the generic points $x$ of irreducible codimension-$1$ subvarieties of $X$ (or equivalently prime divisors on $X$), $H^d(x)$ denotes cohomology of the residue field of $x$, and the map is given componentwise by the composition \[ H^{d+1}(\F(X)) \to H^{d+1}(\F(X)_v) \to H^d(x) ,\] where $v$ is the discrete valuation on $\F(X)$ induced by $x$, the left-hand map is cohomological restriction to the henselisation $\F(X)_v$ and the right-hand map is the map $\partial_v$ from Proposition \ref{prop:local_cohomology}.
\end{theorem}
The residue field of a point $x$ as in the theorem is the function field of the subvariety of $X$ given by the closure of $x$, a finitely generated field extension of $\F$ of transcendence degree $d-1$.
\begin{proof}
  In \cite[§1]{kato}, a complex $C^0_2(X)$ is constructed. Three of the terms of this complex are the following:
  \[ \bigoplus_{x \in X_{d+1}} H^{d+2}(\F(x)) \to \bigoplus_{x \in X_d} H^{d+1}(\F(x)) \to \bigoplus_{x \in X_{d-1}} H^d(\F(x)) \]
  Here $X_j$ denotes the set of points $x \in X$ whose closure has dimension $j$.
  Since $X$ is integral of dimension $d$, there are no points with closure of dimension $d+1$, and precisely one point with closure of dimension $d$.
  The three terms of the complex hence simplify in the following way:
  \[ 0 \to H^{d+1}(\F(X)) \to \bigoplus_{x \in X_{d-1}} H^d(\F(x)) \]
  The transition map on the right-hand side is the one given in the statement of the theorem above.

  It is conjectured in \cite[Conjecture 0.3]{kato} that the complex is exact at the point under consideration.
  This is proven in \cite[Theorem 0.4]{kerzSaito} for $\kar \F$ odd, in \cite[Theorem 0.10]{JannsenHassePrinciples} for $\kar \F=2$ under the assumption of resolution of singularities over finite fields of characteristic two, and in \cite[p.~270]{Suwa_ANoteOnGerstensConjecture} if $\kar \F=2$ and $d \leq 3$.
This proves the claim.
\end{proof}

\begin{theorem}[Kerz--Saito]\label{thm:KerzSaitoMixedChar}
  Let $X$ be a regular integral scheme of dimension $d>0$ proper and flat over $\operatorname{Spec}(\Z[1/2])$ with function field $K$ (automatically of transcendence degree $d-1$ over $\Q$).
  Then the map
  \[ H^{d+1}(K) \to \bigoplus_x H^d(x) \oplus \bigoplus_y H^{d+1}(y) \oplus \bigoplus_z H^{d+1}(z) \]
  is injective.
  Here the first summand is as in Theorem \ref{thm:KerzSaitoFiniteField}, $y$ varies over points of the base change $X_\R$ whose closure has dimension $d-1$, $z$ varies over points of $X_{\Q_2}$ whose closure has dimension $d-1$, and we write $H^d(x)$(and likewise for $y$, $z$) for the cohomology of the residue field of $x$ as before.
  
  The map is given as a direct sum of maps $H^{d+1}(K) \to H^d(x)$ as in Theorem \ref{thm:KerzSaitoFiniteField}, and restriction maps from $H^{d+1}(K)$ into $H^{d+1}(y)$ and $H^{d+1}(x)$.
\end{theorem}
Note that because $X_\R$ and $X_{\Q_2}$ have dimension $d-1$, there are only finitely many points $y$ and $z$ with closure of dimension $d-1$, and for each of these we have an embedding of $K$ into the residue field.

Furthermore, as we are working over $\Z[1/2]$, none of the residue fields occurring have characteristic two.
\begin{proof}
  Similarly to the proof of the preceding theorem, Kato in \cite[Conjecture 0.5]{kato} considers a certain mapping cone complex $\hat{C}_2(X)$ whose term in degree $j$ is the following:
  \[ \bigoplus_{x \in X_j} H^{j+1}(x) \oplus \bigoplus_{x \in (X_{\R})_j} H^{j+2}(x) \oplus \bigoplus_{x \in (X_{\Q_2})_j} H^{j+2}(x) \]
  Since $X$ has dimension $d$ and $X_\R$, $X_{\Q_2}$ have dimension $d-1$,
  the complex is zero in degree $d+1$, and in degree $d$ we only have the contribution of the generic point of $X$ and hence obtain $H^{d+1}(K)$.

  It is now proven in \cite[Theorem 0.5]{kerzSaito}, after \cite[Conjecture 0.5]{kato}, that the complex is exact in degree $d$, proving the claim.
\end{proof}

\begin{proof}[Proof of Theorem \ref{thm:cohom_loc_glob}]
  If $L/K$ is a finite extension of odd degree, then the restriction map $H^{d+1}(K) \to H^{d+1}(L)$ is injective by basic Galois cohomology.
  Assuming that the theorem was true for $L$, for any $\alpha \in H^{d+1}(K)$ as in the hypothesis, the restriction $\operatorname{res}_{L/K}(\alpha)$ to $H^{d+1}(L)$ is non-zero, so by applying the theorem for $L$ there exists a suitable valuation $v$ on $L$ such that $\partial_v\operatorname{res}_{L/K}(\alpha) \in H^d(Lv)$ is non-zero.
  Using the commutative diagram from Proposition \ref{prop:local_cohomology}(3), we find that $\partial_v\alpha \in H^d(Kv)$ is non-zero.
  
  It follows that if the theorem is true for $L$ then it is also true for $K$.
  Hence, by applying Lemma \ref{lem:regular_model_after_odd_extn} and replacing $K$ with $L$, we may assume that there is an integral smooth projective variety $X$ over $\Z[1/2]$ or $\F_p$ with function field $K$.

  In positive characteristic, the statement now follows immediately from Theorem \ref{thm:KerzSaitoFiniteField}.
  In characteristic zero, we consider a cohomology class $\alpha \in H^{d+1}(K)$ which vanishes when restricted to any overfield of $K$ embedding either $\R$ or $\Q_2$.
  In the notation of Theorem \ref{thm:KerzSaitoMixedChar}, this means that the restriction of $\alpha$ to $H^{d+1}(y)$ and $H^{d+1}(z)$ is trivial for all $y$ and $z$.
  Hence the injectivity statement of Theorem \ref{thm:KerzSaitoMixedChar} proves the claim.
\end{proof}

\begin{lemma}
  Let $K$ be finitely generated of Kronecker dimension $d\geq 0$, and let $\alpha \in H^{d+1}(K)$ be non-zero.
  If $\kar K = 0$, assume that $\alpha = (a_0) \cup (a_1) \cup \gamma$ for some $\gamma \in H^{d-1}(K)$ and $a_0, a_1 \in \Q^\times$ such that the Pfister form $\llangle a_0, a_1\rrsquare$ is isotropic over $\Q_2$ and $\R$.
  If $\kar K = 2$, assume that $d \leq 3$ or resolution of singularities holds over finite fields of characteristic two.

  Then there exists a valuation $v$ on $K$ of archimedean rank $d$ such that $\alpha$ is non-zero under the restriction to the henselisation $K_v$, with residue characteristic not two unless $\kar K = 2$.
\end{lemma}
Theorem \ref{thm:pfister_loc_glob} is an immediate consequence.
\begin{proof}
  First consider the case of $\kar K > 0$. We proceed by induction on $d$.
  The case $d=0$ is clear, as we can take the trivial valuation for $v$.

  For $d>0$, Theorem \ref{thm:cohom_loc_glob} shows that there is a valuation $v$ on $K$, with value group $\Z$ and finitely generated residue field $Kv$ of Kronecker dimension $d-1$, such that $\partial_v \alpha$ does not vanish in $H^d(Kv)$. The valuation ring of $v$ is excellent by Lemma \ref{lem:excellent_dvr}.

  By induction hypothesis, there is a valuation $w$ on $Kv$ of rank $d-1$ such that the element $\partial_v\alpha \in H^d(Kv)$ does not vanish when restricted to $(Kv)_w$.
  Consider now the composite valuation $w \circ v$ on $K$, which is of archimedean rank $d$. The henselisation $K_{w \circ v}$ is the unramified extension of $K_v$ with residue field $(Kv)_w$.
  By the compatibility of $\partial_v$ with unramified extensions, Proposition \ref{prop:local_cohomology}(3), the following diagram commutes.
  \[\xymatrix{
      H^{d+1}(K_v) \ar[r] \ar[d] & H^d(Kv) \ar[d] \\
      H^{d+1}(K_{w \circ v}) \ar[r] & H^d((Kv)_w)
  }\]
  As the horizontal arrows are isomorphisms by Proposition \ref{prop:local_cohomology} the restriction of $\alpha$ does not vanish in $H^{d+1}(K_{w \circ v})$, proving the claim.

  In characteristic zero, we follow the same inductive approach.
  There are only two points more to check.
  Firstly, for any $\alpha$ of the given form $(a_0) \cup (a_1) \cup \gamma$, $\alpha$ vanishes when restricted to any overfield which embeds $\R$ or $\Q_2$ since $(a_0) \cup (a_1)$ does (as the Pfister form $\llangle a_0, a_1\rrsquare$ becomes isotropic by assumption), so Theorem \ref{thm:cohom_loc_glob} is applicable.
 Secondly, for $\alpha$ of the same form, if a valuation $v$ on $K$ with value group $\Z$ has residue field of characteristic zero, then $\partial_v(\alpha) = (a_0) \cup (a_1) \cup \partial_v(\alpha)$ again has the required form by Proposition \ref{prop:local_cohomology}(4), so we can continue inductively.
\end{proof}

\section{Building subrings}
\label{sec:building_subrings}

Let $K$ be a finitely generated field of Kronecker dimension $d > 1$.
If $\kar K = 2$, assume that $d \leq 3$ or resolution of singularities holds over finite fields of characteristic two.
Let $K_2 \subseteq K$ be a subfield of Kronecker dimension $d-1$.
In this section we construct a family of subrings of $K$ containing $K_2$.

Let $a_0, \dotsc, a_d \in K^\times$  such that the following conditions are satisfied:
\begin{enumerate}
  \item[(i)] $a_{d-1}, a_d$ are contained in a global subfield $K_1$ of $K_2$;
  \item[(ii)] in characteristic zero, even $a_{d-1}, a_d \in \Q$, and the Pfister form $\llangle a_{d-1}, a_d\rrsquare/\Q$ becomes isotropic over $\Q_2$ and $\R$.
\end{enumerate}
Write $\Delta_0$ for the set of (equivalence classes of) non-trivial valuations of $K_1$ over whose henselisation the Pfister form $\llangle a_{d-1}, a_d\rrsquare$ is anisotropic; this is finite by Lemma \ref{lem:pfister_over_global_fields}, and contains no valuations of mixed characteristic $(0,2)$ by (ii).

We can use Weak Approximation to choose $c \in K_1^\times$ such that the following condition is satisfied.
\begin{enumerate}
  \item[(iii)] For all $v \in \Delta_0$ we have $v(c) = 0$, and the reduction of the polynomial $X^2+X+c$ is irreducible over the residue field $K_1v$.
\end{enumerate}

Abbreviate $q = \llangle a_0, \dotsc, a_d \rrsquare$, and write $\Delta$ for the set of all (equivalence classes of) valuations on $K$ such that $q$ does not become isotropic over the henselisation.

The main result of this section is the following.
\begin{proposition}\label{prop:constructing_ring}
  Under conditions (i, ii, iii) we have
  \[ S_c(q/K) \cdot K_2^\times = \bigcap_{\stackrel{w \in \Delta}{w \text{ trivial on }K_1}} \mathcal{O}_w .\]
\end{proposition}
Here we write $S_c(q/K) \cdot K_2^\times$ for the set products of an element of $S_c(q/K)$ and an element of $K_2^\times$.
Note that the right-hand side in Proposition \ref{prop:constructing_ring} is a subring of $K$ containing $K_2$.

By Lemma \ref{lem:pfister_over_global_fields_special_form} we may assume, by replacing $a_{d-1}, a_d$ by other elements of $K_1$ in a way that preserves (the isomorphism type of) $\llangle a_{d-1}, a_d\rrsquare/K_1$, that we have $v(a_d) \geq 0$ for all $v \in \Delta_0$.
Choose $C \in K_1^\times$ such that $v(C) > 0$ for all $v \in \Delta_0$.
\begin{lemma}\label{lem:for_each_v_exists_a}
  Let $x \in \bigcap_{\stackrel{w \in \Delta}{w \text{ trivial on }K_2}} \mathcal{O}_w$ and $v \in \Delta$ with $v(C) > 0$.
  There exists $a \in K_2^\times$ such that $ax \in \mathfrak{m}_v$.
\end{lemma}
\begin{proof}
  The valuation $v$ is non-trivial on $K_2$ since $v(C) > 0$.
  Let $w$ be the finest coarsening of $v$ trivial on $K_2$, i.e.\ given by quotienting the value group of $v$ by the convex hull of the subgroup $v(K_2^\times)$; this may be trivial.
  Since $K_w$ embeds into $K_v$ and therefore $w \in \Delta$, we have $w(x) \geq 0$ by assumption.
  If $w(x) > 0$, then $v(x) > 0$, so $a=1$ works.
  Otherwise, $w(x) = 0$ and hence $v(x)$ is in the convex hull of $v(K_2^\times)$, so there exists an $a \in K_2^\times$ with $v(ax) > 0$, as desired.
\end{proof}

We now exchange the quantifiers.
\begin{lemma}\label{lem:exists_a_for_all_v}
  Let $x \in \bigcap_{\stackrel{w \in \Delta}{w \text{ trivial on }K_2}} \mathcal{O}_w$.
  There exists $a \in K_2^\times$ such that for all $v \in \Delta$ with $v(C) > 0$ we have $ax \in \mathfrak{m}_v$.
\end{lemma}

To prove this lemma, we use a topological argument.
Consider the space $S_{\mathrm{val}}$ of equivalence classes of valuations on $K$, and endow it with the \emph{constructible topology}, i.e.\ the coarsest topology in which sets of the form $\{ v \colon v(a) \geq 0 \}$ are clopen, where $a \in K$. This makes $S_{\mathrm{val}}$ a totally disconnected compact Hausdorff space, see \cite[Section 1]{HuberKnebusch} or \cite[Section 2]{adfApproximation}.

The use of a topology on the space of valuations on $K$ appears to be new in the current context, although one may wish to compare it to the use of the Zariski topology on the space of valuations in \cite[Section 2]{anArithmeticProofOfPopsTheorem}.

Let $\Delta' = \{ v \in \Delta \colon v(C) > 0 \}$.
\begin{lemma}
  The subset $\Delta' \subseteq S_{\mathrm{val}}$ is closed and hence compact.
  For any $v \in \Delta$, we have $v(c) = 0$, and the polynomial $X^2+X+c$ is irreducible over the residue field of $v$.
\end{lemma}
\begin{proof}
  If $w$ is a valuation on $K$ such that $w(c) \neq 0$, then $w$ restricts to a non-trivial valuation on $K_1$ which is not in $\Delta_0$, and hence $\llangle a_{d-1}, a_d\rrsquare$ and therefore $q$ are isotropic over $K_w$, so $w \not\in \Delta$.
  Hence $\Delta'$ is contained in the clopen set $\{ v \in S_{\mathrm{val}} \colon v(c) = 0 \wedge v(C) > 0 \}$. Similarly, if $K$ is of characteristic zero, then $\Delta'$ is contained in the clopen set $\{ v \in S_{\mathrm{val}} \colon v(1/2) = 0 \}$ consisting of those valuations of residue characteristic not two.

  If $X^2+X+c$ becomes reducible over the residue field of $w$, then by Hensel's Lemma the henselisation $K_w$ embeds $K_1[X]/(X^2+X+c)$, which is unramified of degree two at all places in $\Delta_0$ by construction, so $\llangle a_{d-1}, a_d\rrsquare$ is isotropic over $K_1[X]/(X^2+X+c)$ by Lemma \ref{lem:global_fields_local_global}, hence $\llangle a_{d-1}, a_d\rrsquare$ and therefore $q$ are isotropic over $K_w$, and thus $w \not\in \Delta'$.

  Now let $v \in S_{\mathrm{val}} \setminus \Delta'$. We have to show that there is an entire open neighbourhood of $v$ disjoint from $\Delta'$. This is clear if $v(C) \leq 0$, in particular if $v$ is trivial on $K_1$.
  If the restriction of $v$ to $K_1$ is not in $\Delta_0$, then the same will be true for all valuations in an open neighbourhood of $v$, so this entire open neighbourhood will be disjoint from $\Delta'$.
  In particular, we may assume that $v(a_d) \geq 0$ and that the residue characteristic of $v$ is not two unless $\kar K = 2$.
  Now the criteria for isotropy given in Lemmas \ref{lem:isotropic_in_henselisation_char_not_two} and \ref{lem:isotropic_in_henselisation_char_two} are clearly open conditions.
\end{proof}

\begin{proof}[Proof of Lemma \ref{lem:exists_a_for_all_v}]
  For any given $a \in K_2^\times$, the set of $v \in \Delta'$ such that $ax \in \mathfrak{m}_v$ is open, and for any $v \in \Delta'$ there exists an $a \in K_2^\times$ with $ax \in \mathfrak{m}_v$ by Lemma \ref{lem:for_each_v_exists_a}.
  Hence by compactness of $\Delta'$ there are in fact finitely many $a_1, \dotsc, a_n \in K_2^\times$ such that for each $v \in \Delta'$ there exists an $i$ with $a_i x \in \mathfrak{m}_v$.

  Consider the function $\phi \colon K_2^\times \times K_2^\times \to K_2^\times$ given by $(x, y) \mapsto x^2 + xy + cy^2$.
  For any $v \in \Delta'$, the polynomial $X^2+X+c$ is irreducible over the residue field of $v$, so for any $x \in K_2^\times$ we have $v(\phi(x,1)) = \min(2v(x), 0)$, which by homogeneity implies that for all $x, y \in K_2^\times$ we have $v(\phi(x,y)) = 2\min(v(x), v(y))$.

  Now take $a = \phi(1, \phi(a_1^{-1}, \phi(a_2^{-1}, \dotsc )\dotsb))^{-1}$; this satisfies \[ v(a) \geq \max(v(a_1), \dotsc, v(a_n)) \] for all $v \in \Delta'$ by construction.
  Hence $ax \in \mathfrak{m}_v$ for all $v \in \Delta'$ as desired.
\end{proof}

The technique for the preceding proof is taken from \cite[Lemma 4.2]{adfApproximation}.

\begin{proof}[Proof of Proposition \ref{prop:constructing_ring}]
  The inclusion $\subseteq$ is clear, since for any $w \in \Delta$ trivial on $K_2$ we have $K_2^\times \subseteq \mathcal{O}_w$ and $S_c(q/K) \subseteq S_c(q/K_w) \cap K \subseteq \mathcal{O}_w$ by Lemma \ref{lem:henselian_S_in_valn_ring}.

  For the other inclusion, let $x \in \bigcap_{\stackrel{w \in \Delta}{w \text{ trivial on }K_2}}\mathcal{O}_w$. By Lemma \ref{lem:exists_a_for_all_v}, there exists $a \in K_2^\times$ such that $v(ax) > 0$ for all $v \in \Delta$ with $v(C) > 0$.

  Any $v \in \Delta$ of archimedean rank $d$ must restrict to a non-trivial valuation on $K_1$ since $K/K_1$ has transcendence degree $d-1$. Its restriction to $K_1$ must in fact be a valuation in $\Delta_0$, since otherwise the subform $\llangle a_{d-1}, a_d\rrsquare$ is isotropic over $K_v$.
  In particular, $v(C) > 0$, and so $ax \in \mathfrak{m}_v \subseteq S_c(q/K_v)$ by Proposition \ref{prop:local_S_computation}.

  By Corollary \ref{cor:S_loc_glob} we deduce that $ax \in S_c(q/K)$. Hence
  \[x \in a^{-1} S_c(q/K) \subseteq S_c(q/K) \cdot K_1^\times . \qedhere\]
\end{proof}

\section{Subrings finitely generated over a global field}
\label{sec:subrings_fin_generated}

Let $K$ be finitely generated of Kronecker dimension $d>1$, and if $\kar K = 2$ assume that $d \leq 3$ or resolution of singularities holds over finite fields of characteristic two.
Write $K_0$ for the relative algebraic closure of the prime field in $K$.
In characteristic zero, let $t_1=1$, and in positive characteristic let $t_1 \in K$ be transcendental over the prime field.

Pick a transcendence basis $t_2, \dotsc, t_d$ for $K/K_0(t_1)$ and let $R_0 = K_0[t_1, \dotsc, t_d]$.
Then $K$ is a finite extension of $K_0(t_1, \dotsc, t_d)$, the quotient field of $R_0$.

We wish to construct a finite ring extension of $R_0$ as an intersection of subrings of $K$ as in the last section.
Let $s \in K$ be any element integral over $R_0$.
\begin{definition}
Let \[ R(t_1,t_2, \dotsc, t_d, s) = \bigcap_{c, a_0, \dotsc, a_d \in K, K_2 \subseteq K} S_c(\llangle a_0, \dotsc, a_d\rrsquare/K) \cdot K_2, \]
where the intersection is over those $c, a_0, \dotsc, a_d, K_2$ such that $K_2 \subseteq K$ is relatively algebraically closed and of Kronecker dimension $d-1$ and $S_c(\llangle a_0, \dotsc, a_d\rrsquare/K)\cdot K_2$ is a subring of $K$ containing $R_0$ and $s$.
\end{definition}

\begin{remark}
  We are not requiring all the sets $S_c(\llangle a_0, \dotsc, a_d\rrsquare/K) \cdot K_1$ which we are intersecting to arise as in Section \ref{sec:building_subrings}; in particular, we do not guarantee that all of them are integrally closed, as would be expected from Proposition \ref{prop:constructing_ring}.

  We refrain from imposing this condition to make our later definability result in Proposition \ref{prop:definable} easier to prove.
  (Using the results of \cite{RumelyUndecidabilityGlobalFields}, one can in fact show that the conditions (i,ii,iii) imposed in Section \ref{sec:building_subrings} are first-order definable, but we avoid having to prove this.)
\end{remark}

\begin{lemma}\label{lem:not_integral_exists_bad_divisor}
  For every $x \in K$ not integral over $R_0$ there exists an equicharacteristic valuation $v$ on $K$, with value group isomorphic to $\Z$ and finitely generated residue field of Kronecker dimension $d-1$, such that $x \not\in \mathcal{O}_v \supseteq R_0$. 
\end{lemma}
\begin{proof}
  Consider the minimal polynomial $f$ of $x$ over the quotient field of $R_0$. It does not have coefficients in $R_0$ as otherwise $x$ would be integral over $R_0$.
  Since $R_0$ is a unique factorisation domain, there exists a prime element $y$ of $R_0$ such that $f$ does not even have coefficients in the localisation $S$ of $R_0$ at the prime ideal $(y)$.
  Hence $x$ is not integral over $S$. Observe that $S$ is a discrete valuation ring with residue field $\operatorname{Frac}(R_0/(y))$; this residue field has Kronecker dimension $d-1$ since $\operatorname{Krull dim}(R_0/(y)) = \operatorname{Krull dim}(R_0)-1$ by Krull's Hauptidealsatz.

  Since $x$ is not integral over $S$, there exists a valuation ring $S'$ of $K$ dominating $S$ and not containing $x$ by \cite[Corollary 3.1.4]{englerPrestel}.
  The valuation $v$ associated to $S'$ is as desired.
\end{proof}

The following lemma may be compared to \cite[Fact 1.3(c)]{elemEquivVsIsomI} and \cite[Lemma 5.3]{poonenUniformDefnsInFinGenFields}.
\begin{lemma}\label{lem:exist_Pfister_forms_remaining_anisotropic}
  Let $K_1 = \Q$ or $K_1 = \F_p(t)$, $K_2/ K_1$ a finitely generated separable extension of transcendence degree $d-2$ (i.e.~of Kronecker dimension $d-1$), and $L_2/K_2$ finite and separable.
  There exists a $d$-fold Pfister form $q = \llangle a_1, \dotsc, a_d\rrsquare$ over $K_2$, with $a_{d-1}, a_d \in K_1$, which remains anisotropic over $L_2$.
  In characteristic zero we may even take $a_{d-1}, a_d \in \Q$ such that the form $\llangle a_{d-1}, a_d\rrsquare$ over $\Q$ is isotropic over $\R$ and $\Q_2$.
\end{lemma}
\begin{proof}  
  The statement becomes stronger when we shrink $K_2$, so by considering a separating transcendence basis of $K_2/K_1$ we may assume that $K_2= K_1(s_1, \dotsc, s_{d-2})$ for some transcendental elements $s_i$.

  We use induction on $d$.
  When $d=2$, $K_2 = K_1$ and $L_2$ is a global field.
  By the Chebotarev Density Theorem, there exist infinitely many primes $\mathfrak{p}$ of $K_2$ which are completely split in $L_2$.
  Hence by Lemma \ref{lem:pfister_over_global_fields} we may pick a Pfister form $\llangle a_1, a_2\rrsquare$ over $K_1$ which remains anisotropic over $L_2$ since it is anisotropic over a henselisation of $L_2$.
  In characteristic zero, we may always take $a_1, a_2 \in \Q^\times$ such that $\llangle a_1, a_2\rrsquare$ is anisotropic over $L_2$, but isotropic over $\Q_2$ and $\R$.
  Hence $\llangle a_1, a_2\rrsquare$ is a Pfister form as desired.

  Consider now $d>2$ and write $K_2 = F(s_{d-2})$ for $F=K_1(s_1, \dotsc, s_{d-3})$. There are infinitely many valuations on $K_2$ trivial on $F$ with value group $\Z$ -- it is well-known that these consist of one valuation for every irreducible polynomial in $F[s_{d-2}]$, and one further so-called degree valuation.
  
  Since $L_2/K_2$ is separable, only finitely many of these valuations ramify in $L_2$, so we may pick such a valuation $v$ on $K_2$ unramified in $L_2$; choose an extension to $L_2$ and also denote it by $v$.
  Then the residue field of $v$ is a finite separable extension $E/F$. By induction hypothesis, there exists a $d-1$-fold Pfister form $\llangle a_1, \dotsc, a_{d-1}\rrsquare$ over $F$ which remains anisotropic over $E$.
  Let $t$ be a uniformiser of $v$ over $K_2$, which remains a uniformiser of $L_2$ because $v$ is unramified.
  Then $q = \llangle t, a_1, \dotsc, a_{d-1}\rrsquare$ is a Pfister form over $K_2$ which remains anisotropic over $L_2$ by Lemma \ref{lem:anisotropic_over_dvf}.
\end{proof}

\begin{lemma}\label{lem:R}
  For any $t_1, t_2, \dotsc, t_d, s \in K$ satisfying the conditions above,
  the set $R(t_1, t_2, \dotsc, t_d, s)$ is a subring of $K$ with quotient field containing $K_0(t_1, \dotsc, t_d, s)$.
  Furthermore it is finitely generated as a $K_0$-algebra. 
\end{lemma}
The key result here is that $R(t_1, t_2, \dotsc, t_d, s)$ is finitely generated as a $K_0$-algebra. This is essentially a matter of proving that we are intersecting a family of subrings of $K$ which is sufficiently large.
\begin{proof}
  Write $R = R(t_1, t_2, \dotsc, t_d, s)$.
  The set $R$ is by definition an intersection of subrings of $K$ containing $R_0$ and $s$, and therefore is itself a subring of $K$ containing $R_0[s]$.
  Hence it has quotient field containing $K_0(t_1, \dotsc, t_d, s)$.
  The integral closure of $R_0$ in $K$ is a finite $R_0$-module by \cite[Theorem 4.14]{Eisenbud_CommAlg} (integral domains finitely generated over fields are Japanese).
  If we can show that $R$ is contained in this integral closure, then $R$ is also a finite $R_0$-module since $R_0$ is Noetherian, and therefore $R$ will be a finitely generated $K_0$-algebra since $R_0$ is.

  Hence it remains to show that $R$ is integral over $R_0$.
  Let $x \in K$ be not integral over $R_0$, so by Lemma \ref{lem:not_integral_exists_bad_divisor} there exists an equicharacteristic valuation $v$ on $K$ with valuation ring $\mathcal{O}_v$ containing $R_0$, value group $\Z$ and residue field $Kv$ of Kronecker dimension $d-1$ such that $v(x) < 0$.
  
  If $\kar K = 0$, let $r_1=1$ and $K_1 = \Q$; otherwise, choose an element $r_1 \in Kv$ transcendental over the prime field such that $Kv$ is separable over $K_0(r_1)$, and also write $r_1$ for a fixed lift of $r_1$ in $K$. In this way, $K_1 = K_0(r_1)$ is a subfield of $K$ which is identified with a subfield of $Kv$.
  By lifting a transcendence basis of $Kv/K_1$, we can furthermore find a subfield $K_2 \supseteq K_1$ of $K$ of Kronecker dimension $d-1$ on which $v$ is trivial. We may assume that $K_2$ is relatively algebraically closed in $K_2$ by passing to the relative algebraic closure. 
  
  Consider the ring $R_2 = R_0 K_2$. This is a finitely generated $K_2$-algebra with fraction field of transcendence degree $1$ over $K_2$, so by the Noether Normalisation Lemma we can choose $t \in R_2$ with $R_2$ integral over $K_2[t]$.
  Since $K$ is a finite field extension of $K_2(t)$, there are only finitely many valuations $w_1, \dotsc, w_k$ on $K_2(t)$ which are trivial on $K_2$ and extend the degree valuation on $K_2(t)$, i.e.\ have $w_i(t) < 0$.
  All other valuations on $K$ trivial on $K_2$ necessarily have valuation ring containing $K_2[t]$ and therefore containing $R_0$ by integral closedness of valuation rings.
  
  By Lemma \ref{lem:exist_Pfister_forms_remaining_anisotropic} there exists an anisotropic Pfister form $\llangle a_1, \dotsc, a_d\rrsquare$ over $Kv$ with $a_{d-1},a_d \in K_1$.
  We may lift the coefficients $a_i$ to elements $b_i \in K$ while maintaining that $b_{d-1}, b_d \in K_1$; in characteristic zero, we may furthermore take $\llangle b_{d-1}, b_d\rrsquare$ to be isotropic over $\Q_2$ and $\R$.

  Pick $b_0 \in K^\times$ which is a uniformiser for $v$, and has $w_i(b_0-1) >0$ for the $w_i$ from above; this is possible by weak approximation for finitely many independent valuations.
  The Pfister form $q = \llangle b_0, \dotsc, b_d\rrsquare$ over $K$ is anisotropic over the henselisation $K_v$ by Lemma \ref{lem:anisotropic_over_dvf}.
  Furthermore $q$ is isotropic over the henselisations $K_{w_i}$, since even the subform $\llangle b_0, b_d\rrsquare$ is isotropic over these henselisations by choice of $b_0$ and Lemmas \ref{lem:isotropic_in_henselisation_char_not_two} and \ref{lem:isotropic_in_henselisation_char_two}.
  
  By Proposition \ref{prop:constructing_ring}, we can choose $c \in K_1^\times$ such that $S_c(q/K) \cdot K_2^\times$ is an integrally closed subring of $K$ contained in $\mathcal{O}_v$ (and hence not containing $x$) and containing $K_2[t]$, and therefore containing $R_0$ and even $R_0[s]$ by integral closedness.
  This proves the claim.
\end{proof}

\section{Definability, biinterpretation and axiomatisability}

As in the last section, let $K$ be a finitely generated field of Kronecker dimension $d>1$, and if $\kar K = 2$ assume that $d \leq 3$ or resolution of singularities holds over finite fields of characteristic two.
Write $K_0 \subseteq K$ for the relative algebraic closure of the prime field as before.

Recall that in the last section we worked with elements $t_1, \dotsc, t_d, s \in K$ such that $K$ is a finite extension of $K_0(t_1, \dotsc, t_d)$, $t_1 = 1$ if $\kar K = 0$, and $s$ is integral over $K_0[t_1, \dotsc, t_d]$.

\begin{proposition}\label{prop:definable}
  The subring $R(t_1, \dotsc, t_d, s)$ is first-order definable in $K$ in the language of rings in terms of $t_1, \dotsc, t_d, s$.
  In other words, there exists a first-order formula $\psi(X, T_1, \dotsc, T_d, S)$ in the language of rings such that we have \[ R(t_1, \dotsc, t_d, s) = \{ x \in K \colon K \models \psi(x, t_1, \dotsc, t_d, s) \} \]
  for any $t_1, \dotsc, t_d, s \in K$ as above.
  The defining formula $\psi$ only depends on $d$, and is otherwise independent of $K$.
\end{proposition}
\begin{proof}
  By \cite[Theorem 1.1]{poonenUniformDefnsInFinGenFields}, there is a sentence which distinguishes finitely generated fields of characteristic zero from finitely generated fields of positive characteristic.
  This allows us to construct our formula $\psi$ separately in the cases of characteristic zero and of positive characteristic.
  
  By \cite[Theorem 1.4]{poonenUniformDefnsInFinGenFields}, for every $n \geq 1$ there exists a formula $\psi_n(X_1, \dotsc, X_n)$ such that for any $x_1, \dotsc, x_n \in K$ we have $K \models \psi_n(x_1, \dotsc, x_n)$ if and only if the $x_i$ are algebraically dependent over the prime field.

  This means that we can define the subfield $K_0 \subseteq K$ as the set of $x \in K$ such that $K \models \psi_1(x)$.

  Similarly, for any $e\geq 0$ the relatively algebraically closed subfields $K_2 \subseteq K$ of transcendence degree $e$ over the prime field are exactly the sets of the form $\{ x \in K \colon K \models \psi_{e+1}(x, x_1, \dotsc, x_{e}) \}$, where the $x_1, \dotsc, x_{e} \in K$ satisfy $K \models \neg\psi_{e}(x_1, \dotsc, x_{e})$.
  Hence the family of relatively algebraically closed subfields $K_2 \subseteq K$ of Kronecker dimension $d-1$ is definable.

  By inspection of the definition of $R(t_1, \dotsc, t_d, s)$, it therefore suffices to show that the predicate $S$ is definable, i.e.\ there is a formula $\varphi(X, C, A_0, \dotsc, A_d)$ such that $K \models \varphi(x, c, a_0, \dotsc, a_d)$ if and only if $x \in S_c(\llangle a_0, \dotsc, a_d\rrsquare/K)$.

  Such a formula, even existential, can be read off from Definition \ref{defn:S} and Remark \ref{rem:alternative_defn_S'}:
  We have $x \in S_c(\llangle a_0, \dotsc, a_d\rrsquare/K)$
  if and only if there exist elements $x_1, \dotsc, x_{2^{d+1}} \in K[X]/(X^2+(1-x)X+c)$ generating the unit ideal in $K[X]/(X^2+(1-x)X+c)$ and which are a zero of $\llangle a_0, \dotsc, a_d\rrsquare$.
  By identifying elements of $K[X]/(X^2+(1-x)X+c)$ with pairs of elements of $K$, e.g.\ identifying $(b_1, b_2) \in K^2$ with the image of the polynomial $b_1X + b_2$ in $K[X]/(X^2+(1-x)X+c)$, and then expanding addition and multiplication in $K[X]/(X^2+(1-x)X+c)$ in terms of pairs $(b_1, b_2)$, this is seen to be an existential first-order property.
\end{proof}

We now take $t_1, t_2, \dotsc, t_d, s$ such that $K$ is the quotient field of $K_0[t_1, \dotsc, t_d, s]$; this is possible, since we may arrange for $K/K_0(t_1, \dotsc, t_d)$ to be separable and hence have a primitive element $s$:
this is automatic in characteristic zero, and in characteristic $p>0$ we may take $t_1, \dotsc, t_d$ to be a separating transcendence basis of $K$ over the prime field $\F_p$ since $\F_p$ is perfect.
The primitive element $s$ for $K/K_0(t_1, \dotsc, t_d)$ can always be taken to be integral over $K_0[t_1, \dotsc, t_d]$, since multiplying $s$ by a non-zero element of $K_0[t_1, \dotsc, t_d]$ does not change $K_0(t_1, \dotsc, t_d, s)$.

Now $R(t_1, \dotsc, t_d, s)$ is definable by Proposition \ref{prop:definable} and has quotient field $K$ by Lemma \ref{lem:R}.
Hence we have proven Theorem \ref{thm:intro_exists_small_definable_subring} from the introduction.

Let us remark that the deduction of a positive answer to Pop's question from the definability of certain subrings of $K$ has been known since at least \cite{Scanlon_PopsQuestion} (unaffected by the error \cite{Scanlon_PopsQuestionErratum} therein).
Here, we establish Corollary \ref{cor:intro_biinterpretable_quasiaxiomatisable} by using techniques from \cite{AschenbrennerKhelifNaziazenoScanlon} closely related to \cite{Scanlon_PopsQuestion}, in particular the notion of bi-interpretability of structure discussed in Section 2 thereof. 
Whenever we speak of definability, interpretability or bi-interpretability, we always allow parameters.
Let us fix $t_1, \dotsc, t_d, s$ as above and write $R = R(t_1, \dotsc, t_d, s) \subseteq K$.

\begin{lemma}\label{lem:K_R_biinterpretable}
  The rings $K$ and $R$ are bi-interpretable.
\end{lemma}
\begin{proof}
  We established above that $R$ is definable (and hence interpretable) in $K$. Since $K$ is the quotient field of $R$, we can interpret $K$ in $R$ as the set of pairs of elements $\{ (a, b) \in R^2 \colon b \neq 0 \}$ under the obvious equivalence relation, with definable addition and multiplication.
  In this way, the element $r \in R$ is identified with the (equivalence class of) the pair $(r, 1)$ in the interpreted copy of $K$. In the converse direction, any $x \in K$ is identified with the equivalence class $\{ (a, b) \in R^2 \colon b \neq 0, x = a/b \}$, which is definable in $K$, uniformly in $x$. 
  This means that the pair of interpretations given is in fact a bi-interpretation.
\end{proof}

\begin{lemma}\label{lem:R_Z_biinterpretable}
  The ring $R$ is bi-interpretable with $\Z$.
\end{lemma}
\begin{proof}
  Recall that $R$ is a finitely generated algebra over the prime field of $K$.
  In particular, if $\kar K \neq 0$, then $R$ is a finitely generated ring which is a domain, and the result follows from \cite[Theorem 3.1]{AschenbrennerKhelifNaziazenoScanlon}.
  In characteristic zero, this result cannot be invoked as written, since $\Q$ is not a finitely generated ring; the result needed is the following proposition.
\end{proof}

\begin{proposition}
  Let $A$ be an integral domain which is finitely generated over $\Q$ or $\F_p(t)$ for some prime $p$. Then $A$ is bi-interpretable with $\Z$.
\end{proposition}
\begin{proof}
  Write $D$ for the relative algebraic closure of the ground field in $A$; by using Poonen's predicates $\psi_n$ in the quotient field of $A$ as in the proof of Proposition \ref{prop:definable}, we see that $D$ is a global field definable in $A$.
  By \cite{RumelyUndecidabilityGlobalFields}, Gödel functions for all finite sequences are definable in $D$, and hence $D$ is bi-interpretable with $\Z$.
  The remainder of the proof is exactly the same as \cite[ Theorem 3.1 (pp.~33-34)]{AschenbrennerKhelifNaziazenoScanlon}, where integral domains $A$ are considered which are finitely generated over some finitely generated integral domain $D$ whose quotient field is a global field.
\end{proof}

\begin{corollary}\label{cor:final_biinterpretable_quasiaxiomatisable}
  The ring $K$ is bi-interpretable with $\Z$. There exists a formula $\varphi_K$ in the language of rings such that for any finitely generated field $L$ we have $L \models \varphi_K$ if and only if $L$ is isomorphic to $K$.
\end{corollary}
\begin{proof}
  Since bi-interpretability is an equivalence relation, the first part follows immediately from Lemmas \ref{lem:K_R_biinterpretable} and \ref{lem:R_Z_biinterpretable}.
  For the second part, we may as well consider fields in an enriched language containing -- beside the usual symbols from the language of rings -- a unary function symbol $\operatorname{inv}$, to be interpreted as $\operatorname{inv}(x) = x^{-1}$ for every non-zero element $x$ and $\operatorname{inv}(0) = 0$, since this is obviously definable in terms of the other symbols.
  Now finitely generated fields are in fact finitely generated structures in this language -- this is precisely what it means to be a finitely generated field --, so the result follows from \cite[Proposition 2.28]{AschenbrennerKhelifNaziazenoScanlon}.
\end{proof}

\bibliographystyle{amsalpha}
\bibliography{../Bibliography}

\providecommand{\bysame}{\leavevmode\hbox to3em{\hrulefill}\thinspace}
\providecommand{\MR}{\relax\ifhmode\unskip\space\fi MR }
\providecommand{\MRhref}[2]{%
  \href{http://www.ams.org/mathscinet-getitem?mr=#1}{#2}
}
\providecommand{\href}[2]{#2}
\begin{thebibliography}{EKM08}

\bibitem[ADF19]{adfApproximation}
Sylvy Anscombe, Philip Dittmann, and Arno Fehm, \emph{Approximation theorems
  for spaces of localities}, preprint, available as arXiv:1901.02632 [math.AC],
  2019.

\bibitem[AKNS]{AschenbrennerKhelifNaziazenoScanlon}
Matthias Aschenbrenner, Anatole Khélif, Eudes Naziazeno, and Thomas Scanlon,
  \emph{The logical complexity of finitely generated rings}, Int. Math. Res.
  Notices, to appear.

\bibitem[CP09]{CossartPiltant_ResolutionOfSingularitiesOfThreefoldsII}
Vincent Cossart and Olivier Piltant, \emph{Resolution of singularities of
  threefolds in positive characteristic {II}}, Journal of Algebra \textbf{321}
  (2009), 1836--1976.

\bibitem[Dit18a]{definingSubringsOfFinGenFieldsOfCharNotTwo}
Philip Dittmann, \emph{Defining subrings in finitely generated fields of
  characteristic not two}, manuscript, available as arXiv:1810.09333v1
  [math.LO], 22 October 2018.

\bibitem[Dit18b]{DittmannThesis}
\bysame, \emph{A model-theoretic approach to the arithmetic of global fields},
  doctoral thesis, University of Oxford, 2018.

\bibitem[Eis04]{Eisenbud_CommAlg}
David Eisenbud, \emph{Commutative algebra with a view toward algebraic
  geometry}, Springer, 2004.

\bibitem[EKM08]{ElmanKarpenkoMerkurjev_AlgGeomTheoryOfQuadForms}
Richard Elman, Nikita Karpenko, and Alexander Merkurjev, \emph{The algebraic
  and geometric theory of quadratic forms}, American Mathematical Society,
  2008.

\bibitem[EP05]{englerPrestel}
Antonio~J. Engler and Alexander Prestel, \emph{Valued fields}, Springer, 2005.

\bibitem[GS17]{centralSimpleAlgebrasAndGalCohom}
Philippe Gille and Tamás Szamuely, \emph{Central simple algebras and {Galois}
  cohomology}, second ed., Cambridge University Press, 2017.

\bibitem[HK94]{HuberKnebusch}
Roland Huber and Manfred Knebusch, \emph{On valuation spectra}, Contemporary
  Mathematics \textbf{155} (1994), 167--206.

\bibitem[IT14]{travauxGabberX}
Luc Illusie and Michael Temkin, \emph{{Exposé X. Gabber's modification theorem
  (log smooth case)}}, Travaux de {G}abber sur l'uniformisation locale et la
  cohomologie étale des schémas quasi-excellents (Luc Illusie, Yves Laszlo,
  and Fabrice Orgogozo, eds.), Astérisque, vol. 363--364, Société
  Mathématique de France, 2014, pp.~167--212.

\bibitem[Jan16]{JannsenHassePrinciples}
Uwe Jannsen, \emph{Hasse principles for higher-dimensional fields}, Ann. Math.
  \textbf{183} (2016), no.~1, 1--71.

\bibitem[Kat86]{kato}
Kazuya Kato, \emph{A {Hasse} principle for two-dimensional global fields}, J.
  reine angew. Math. \textbf{366} (1986), 142--180.

\bibitem[KS12]{kerzSaito}
Moritz Kerz and Shuji Saito, \emph{Cohomological {Hasse} principle and motivic
  cohomology for arithmetic schemes}, Publ. Math. IHES \textbf{115} (2012),
  no.~1, 123--183.

\bibitem[Liu02]{LiuAlgGeomAndArithCurves}
Qing Liu, \emph{Algebraic geometry and arithmetic curves}, Oxford University
  Press, 2002.

\bibitem[Neu92]{Neukirch}
Jürgen Neukirch, \emph{{Algebraische Zahlentheorie}}, Springer, 1992.

\bibitem[NSW08]{NeukirchSchmidtWingberg}
Jürgen Neukirch, Alexander Schmidt, and Kay Wingberg, \emph{Cohomology of
  number fields}, second ed., Springer, 2008.

\bibitem[Poo07]{poonenUniformDefnsInFinGenFields}
Bjorn Poonen, \emph{Uniform first-order definitions in finitely generated
  fields}, Duke Math. J. \textbf{138} (2007), no.~1, 1--21.

\bibitem[Poo09]{poonenUniversalExistential}
\bysame, \emph{Characterizing integers among rational numbers with a
  universal-existential formula}, Amer. J. Math. \textbf{131} (2009), no.~3,
  675--682.

\bibitem[Poo17]{poonenRationalPoints}
\bysame, \emph{Rational points on varieties}, American Mathematical Society,
  2017.

\bibitem[Pop02]{elemEquivVsIsomI}
Florian Pop, \emph{Elementary equivalence versus isomorphism}, Inventiones
  math. \textbf{150} (2002), 385--408.

\bibitem[Pop17]{elemEquivVsIsomII}
\bysame, \emph{Elementary equivalence versus isomorphism {II}}, Algebra \&
  Number Theory \textbf{11} (2017), 2091--2111.

\bibitem[Pop18]{PopDistinguishingEveryFinitelyGeneratedFieldOfCharNotTwo}
\bysame, \emph{Distinguishing \noopsort{b}every finitely generated field of
  characteristic $\neq 2$ by a single field axiom}, preprint, available as
  arXiv:1809.00440v2 [math.AG], 14 December 2018.

\bibitem[Rum80]{RumelyUndecidabilityGlobalFields}
Robert~S.\ Rumely, \emph{Undecidability and definability for the theory of
  global fields}, Trans. Amer. Math. Soc. \textbf{262} (1980), no.~1, 195--217.

\bibitem[Sca08]{Scanlon_PopsQuestion}
Thomas Scanlon, \emph{Infinite finitely generated fields are biinterpretable
  with $\mathbb{N}$}, J. Amer. Math. Soc. \textbf{21} (2008), no.~3, 893--908.

\bibitem[Sca11]{Scanlon_PopsQuestionErratum}
\bysame, \emph{Erratum to ``{Infinite} finitely generated fields are
  biinterpretable with $\mathbb{N}$''}, J. Amer. Math. Soc. \textbf{24} (2011),
  no.~3, 917.

\bibitem[Ser97]{cohomologieGaloisienne}
Jean-Pierre Serre, \emph{Cohomologie galoisienne}, Springer, 1997, cinquième
  édition.

\bibitem[Spi96]{anArithmeticProofOfPopsTheorem}
Michael Spiess, \emph{An arithmetic proof of {Pop's Theorem} concerning
  {Galois} groups of function fields over number fields}, J. reine angew. Math.
  \textbf{478} (1996), 107--126.

\bibitem[Suw95]{Suwa_ANoteOnGerstensConjecture}
Noriyuki Suwa, \emph{A note on {Gersten's} conjecture for logarithmic
  {Hodge--Witt} sheaves}, $K$-theory \textbf{9} (1995), no.~3, 245--271.

\bibitem[Wad83]{WadsworthPHenselianFieldsKTheoryGalCohomGWittRings}
Adrian~R.\ Wadsworth, \emph{$p$-henselian fields: {$K$}-theory, {Galois}
  cohomology and graded {Witt} rings}, Pacific J. Math. \textbf{105} (1983),
  no.~2, 473--496.

\end{thebibliography}

\end{document}